\documentclass{amsart}
\usepackage[noadjust]{cite}
\usepackage{tikz,amsmath,amssymb,enumerate,amsfonts,stackengine,multicol,microtype,epsfig,xcolor,fourier}
\usepackage[inline]{enumitem}
\usepackage[mathic=true]{mathtools}
\usetikzlibrary{decorations.markings,decorations.fractals}
\usepackage[hidelinks=true]{hyperref}

\theoremstyle{plain}
\newtheorem{theorem}{Theorem}[section]

\newtheorem{proposition}[theorem]{Proposition}
\newtheorem{lemma}[theorem]{Lemma}
\newtheorem{corollary}[theorem]{Corollary}

\theoremstyle{definition}

\newtheorem{definition}[theorem]{Definition}
\newtheorem{remark}[theorem]{Remark}

\newcommand{\occult}[1]{}

\newcommand{\annotation}[1]{\leavevmode\raise.8ex\hbox to0pt{\hss\ensuremath{\curlyvee}\hss}\marginpar{\tiny\baselineskip6pt #1}}

\usepackage{soul}

\newcommand{\id}{\mathrm{id}}

\stackMath
\makeatletter
\newcommand*{\relbarfill@}{\arrowfill@\relbar\relbar\relbar}
\newcommand*{\xminus}[2][]{\ext@arrow 0055\relbarfill@{#1}{#2}}
\newcommand\xxrightarrow[2][]{\mathrel{%
  \setbox2=\hbox{\stackon{\scriptscriptstyle#1}{\scriptscriptstyle#2}}%
  \stackunder[-2pt]{%
    \stackon[.2pt]{\xminus{\makebox[\dimexpr\wd2\relax]{}}}{\scriptscriptstyle#2}%
  }{%
   \scriptscriptstyle#1\,%
  }\,\llap{$\to$}%
}}
\makeatother
\newcommand{\lto}[1]{\xxrightarrow[\;#1\;]{}}
\newcommand{\dfn}{\mathbin{{\vcentcolon}\hskip-1pt{=}}}
\let\emptyset\varnothing
\newcommand\R{\ensuremath{\mathbb R}}

\newcommand\N{\ensuremath{\mathbb N}}
\newcommand\Oc{\ensuremath{\mathcal O}}

\makeatletter
\newcommand*{\Relbarfill@}{\arrowfill@\Relbar\Relbar\Relbar}
\newcommand*{\xeq}[2][]{\ext@arrow 0055\Relbarfill@{#1}{#2}}
\newcommand\xequal[2][]{\mathbin{\ext@arrow 0055{\Equalfill@}{#1}{#2}}}
\def\Equalfill@{\arrowfill@=\Relbar=}
\makeatother

\def\norm#1{\|#1\|}
\def\conorm#1{\|\!{\lfloor}#1{\rfloor}\!\|}
\def\normconorm#1{\conorm{#1}\le\norm{#1}}

\begin{document}

\title[Accessibility and centralizers for flows]{Accessibility and centralizers for partially hyperbolic flows}

\begin{abstract}
Stable accessibility for partially hyperbolic diffeomorphisms is central to their ergodic theory, and we  establish its \(C^1\)-density among \textbullet~all, \textbullet~volume-preserving, \textbullet~symplectic, and \textbullet~contact partially hyperbolic \emph{flows}. 

As applications, we obtain in each of these 4 categories  \(C^1\)-density of \emph{\(C^1\)-stable} topological transitivity, ergodicity, and triviality of the centralizer.
\end{abstract}

\author{Todd Fisher}
\address{T.~Fisher, Department of Mathematics, Brigham Young University, Provo, UT 84602, USA}\email{tfisher@math.byu.edu}
\author{Boris Hasselblatt}
\address{B.~Hasselblatt, Department of Mathematics, Tufts University, Medford, MA 02144, USA}\email{boris.hasselblatt@tufts.edu}

\thanks{T.F.\ is supported by Simons Foundation grant \# 580527.}


\maketitle

\section{Introduction}
The renaissance of partial hyperbolicity that began in the 1990s centered on the quest for stably ergodic dynamical systems. The Hopf argument as the central technical device brought the notion of accessibility to the fore, and this motivated results to the effect that stable accessibility of partially hyperbolic diffeomorphisms (Definition \ref{DEFparthypnarrow}) is  \(C^1\)-dense \cite[Main Theorem]{DolgopyatWilkinson}, \cite{AvilaCrovisierWilkinson}, \cite[Theorem 8.5]{PesinPHLectures}. Our first aim is to show that this also holds for flows. Among the applications is a \(C^1\)-open dense set of partially hyperbolic flows with trivial flow-centralizer.
\subsection*{Statement of results}
\begin{theorem}[Generic accessibility]\label{THMOpenDenseAccessibility}
For any smooth compact manifold \(M\) and \(r\ge1\), \(C^1\)-stable accessibility (Definition \ref{DEFAccessibility})
is \(C^1\) dense among \begin{itemize*}[before=\quad,itemjoin=\quad]
\item all,
\item volume-preserving,
\item symplectic, and
\item contact
\end{itemize*}\quad
partially hyperbolic \(C^r\) flows on
\(M\).
\end{theorem}
Our motivation for establishing genericity of accessibility was to adapt arguments of Burslem \cite{Burslem} in order to establish generic triviality of the flow-centralizer (Definition \ref{DEFTrivialCentralizer}), in particular a relative paucity of faithful \(\mathbb R^k\) Anosov actions.
\begin{theorem}[No centralizer]\label{THMTrivFlowCentralizer}
On any smooth compact manifold \(M\) and for any \(r\ge1\), \(C^r\)-flows which \(C^1\)-stably have trivial flow-centralizer are \(C^1\) dense among 
\begin{itemize*}[before=\quad,itemjoin=\quad]
\item all,
\item volume-preserving,
\item symplectic, and
\item contact
\end{itemize*}\quad
partially hyperbolic \(C^r\) flows on \(M\).
\end{theorem}
There are more direct applications of accessibility:
\begin{corollary}[Generic transitivity]\label{CORGenericTransitivity}
On any smooth compact manifold \(M\) and for any \(r\ge1\), the set of topologically transitive \(C^r\)-flows is \(C^1\) open and dense among 
\begin{itemize*}[before=\quad,itemjoin=\quad]
\item volume-preserving
\item symplectic
\item contact
\end{itemize*}\quad
partially hyperbolic \(C^r\) flows on \(M\).
\end{corollary}
\begin{proof}
Theorem \ref{THMOpenDenseAccessibility} provides an open dense set of accessible such flows; their time-1 maps are accessible volume-preserving partially hyperbolic diffeomorphisms for which almost every point has a dense orbit \cite{Brin}, \cite[Theorem 8.3]{PesinPHLectures}; those points then have dense flow-orbits.
\end{proof}
Indeed, strong ergodic properties, such as the K-property \cite[Definition 3.4.2]{FH}, are similarly common if one adds the assumption of center-bunching \cite{BurnsWilkinson}:
\begin{corollary}[K-property]\label{COR}
On any smooth compact manifold \(M\) and for any \(r\ge2\), the set of \(C^r\)-flows for which the natural volume has the K-property (for all time-$t$ maps) is \(C^1\) open and dense among 
\begin{itemize*}[before=\quad,itemjoin=\quad]
\item volume-preserving
\item symplectic
\item contact
\end{itemize*}\quad
center-bunched partially hyperbolic \(C^r\) flows on \(M\).
\end{corollary}
\begin{proof}
Theorem \ref{THMOpenDenseAccessibility} provides an open dense set of accessible such flows; for any \(t\neq0\) their time-$t$ maps are accessible volume-preserving center-bunched partially hyperbolic diffeomorphisms and hence have the K-property \cite[Theorem 0.1]{BurnsWilkinson}.
\end{proof}
As these corollaries indicate, accessibility is central to the ergodic theory of partially hyperbolic dynamical systems (Definition \ref{DEFparthypnarrow}), the theory of which was revived 2 decades after its founding \cite{BrinPesinAnnPH,BrinPesinPH} when Pugh and Shub sought nonhyperbolic examples of volume-preserving dynamical systems that are \emph{stably} ergodic \cite{PughShub1,PughShub2,PughShub3,PughShub4}. Stable ergodicity was established by adapting the Hopf argument \cite[\S 7.1,2]{FH}, and accessibility is the key ingredient \cite{BurnsWilkinson}.

In this context, flows (as opposed to diffeomorphisms) have received little attention: accessibility of a partially hyperbolic flow and of its time-1 map are equivalent, so a theory for diffeomorphisms suffices for establishing ergodicity for partially hyperbolic flows; indeed, the initial examples of stably ergodic partially hyperbolic diffeomorphisms were time-1 maps of hyperbolic flows. Moreover, the \emph{stability} of ergodicity and accessibility of the time-1 map implies their stability for the flow.

Considering flows becomes salient, however, when investigating the \emph{prevalence} of (stable) accessibility and ergodicity. Diffeomorphisms are rarely the time-1 map of a flow, so density or genericity results for diffeomorphisms do not automatically imply like results for flows. The issue is that one needs to argue that the flows themselves rather than just their time-1 maps can be perturbed in a desired fashion.
\subsection*{Centralizers}
Our interest in accessibility arose from a desire to understand the centralizer of flows beyond hyperbolic ones \cite{BakkerFisherHasselblatt,FH}. Partially hyperbolic flows are a natural next step, and this led us to wanting to adapt the pertinent result \cite[Theorem 1.2]{Burslem} to flows---and its proof uses accessibility in an essential way.

The centralizer of a dynamical system reflects the symmetries of that system, and this leads to the expectation that the centralizer of a (sufficiently complex, see \cite[Example 1.8.8]{FH}) dynamical system is often small. (It also reflects nonuniqueness of conjugacies \cite[p.\ 97]{FH}.) Since the notion is relative to an ambient group \cite[Remark 1.8.9]{FH}, we make the needed terminology explicit.
\begin{definition}\label{DEFTrivialCentralizer}
The \emph{flow-centralizer} of a \(C^r\) flow consists of the \(C^r\) flows that commute with it, and there are different types of triviality \cite{LOS18}. We say that the centralizer is \emph{trivial} if it consists of constant scalings of the flow, i.e., the generating vector fields of commuting flows are constant scalings of the given vector field; in this case we also say that the flow is \emph{self-centered}. 

A flow has \emph{quasi-trivial} flow-centralizer if commuting vector fields are related to the given one by a smooth scalar factor. 
\end{definition}
\begin{remark}[{Hyperbolicity implies small centralizer \cite[\S9.1]{FH}}]
An Anosov flow has trivial flow-centralizer \cite[Corollary 9.1.4]{FH}, and this extends to kinematic-expansive flows on a connected space with at most countably many chain-components, all of which are topologically transitive \cite[Theorem 9.1.3]{FH}. 

\emph{Quasi}-triviality of the centralizer holds for (Bowen-Walters) expansive flows \cite{Oka76} and indeed \(C^r\)-generic flows \cite{LOS18} (including volume-preserving ones \cite{BV19}).

An open and dense subset of \(C^\infty\) Axiom-A flows with a strong transversality condition has (properly) trivial flow-centralizer \cite{Sad79}, as do transitive Komuro expansive flows \cite{BRV18} (this includes the Lorenz attractor), and \(C^1\)-generic sectional Axiom-A flows \cite{BV18,BCW09}.

Indeed, hyperbolic flows usually have small centralizers \cite[Theorem 9.1.3]{FH}, and we extend
this and the requisite accessibility result to partially hyperbolic flows. 

One can also look for diffeomorphisms that commute with a flow; the set of these is the \emph{diffeomorphism-centralizer} of the flow.  Even Anosov flows can have nontrivial diffeomorphism-centralizers \cite[Section 5]{Obata}.
\end{remark}
While there has been interesting work beyond the hyperbolic context (the centralizer is quasi-trivial for a \(C^1\)-generic flow with at most finitely many sinks or sources \cite{Obata}, trivial if the flow moreover has at most countably many chain-recurrence classes), our results produce \emph{open} dense sets with the desired properties, whereas elsewhere, often only ``residual'' is known.

If one thinks of the centralizer question as the possibility of embedding a flow in a faithful \(\mathbb R^2\)-action (or a diffeomorphism into a faithful \(\mathbb Z^2\)-action), then a deeper probe could focus on the classification (or rigidity) of \(\mathbb R^k\)-actions for \(k\ge2\), e.g., aiming to show that they are necessarily algebraic. Great efforts have already been devoted to this aim \cite{KatokSpatzier5,DamjanovicXu}, and quite recently, these have been pushed into the partially hyperbolic realm---for discrete time (for smooth, ergodic perturbations of certain algebraic systems, the smooth centralizer is either virtually \(\mathbb Z^l\) or contains a smooth flow \cite{DamjanovicWilkinsonXu,BarthelmeGogolev}). Also, the centralizer of a partially hyperbolic \(\mathbb T^3\)-diffeomorphism homotopic to an Anosov automorphism is virtually trivial unless the diffeomorphism is smoothly conjugate to its linear part \cite{GanShiZhang}.\\

\subsection*{Acknowledgements}
We would like to thank Sylvain Crovisier, Jana Rodr\'iguez Hertz, Raul Ures, and Amie Wilkinson for helpful suggestions.

\section{Background}

In this section we review needed definitions and previous results.   \cite{FH} provides much of the basic background, but we define two main notions here.
\begin{definition}[Partially hyperbolic]\label{DEFparthypnarrow}
An \emph{embedding} $f$ is said to be \emph{(strongly) partially hyperbolic} on  a compact $f$-invariant set \(\Lambda\)\index{partially hyperbolic} if there exist
numbers $C>0$,
\begin{equation}\label{splitting23}
0<\lambda_1\le\mu_1<\lambda_2\le\mu_2 <\lambda_3\le\mu_3\text{ with }\mu_1<1<\lambda_3
\end{equation}
and an invariant splitting into nontrivial
stable\index{distribution!stable}\index{stable!distribution},
central\index{distribution!central} and
unstable\index{distribution!unstable}\index{unstable!distribution} subbundles 
\begin{equation}\label{splitting21}
T_xM=E^s (x)\oplus E^c(x)\oplus E^u (x),\quad
d_xf E^\tau(x)=E^\tau(f(x)),\ \tau=s,c,u
\end{equation}
such that 
if $n\in\N$, then
\begin{align*}
C^{-1}\lambda_1^n\le\normconorm{d_xf^n\restriction E^s(x)}&\le C\mu_1^n,\\
C^{-1}\lambda_2^n\le\normconorm{d_xf^n\restriction E^c(x)}&\le C\mu_2^n,\\
C^{-1}\lambda_3^n\le\normconorm{d_xf^n\restriction E^u(x)}&\le C\mu_3^n.
\end{align*}
In this case we set $E^{cs}\dfn E^c\oplus E^s$ and $E^{cu}\dfn E^c\oplus E^u$.

A \emph{flow} is said to be partially hyperbolic on a compact flow-invariant set if its time-1 map is partially hyperbolic on it, and \emph{uniformly hyperbolic} if the center direction of the time-1 map consists only of the flow direction. In either case we say that a dynamical system is partially hyperbolic on an invariant set.
\end{definition}
For a partially hyperbolic set $\Lambda$ and $x\in \Lambda$ there exist local stable and unstable manifolds that define global stable and unstable manifolds denoted by 
$W^s(x)$ and $W^u(x)$ respectively. 
\begin{remark}[Persistence of partial hyperbolicity]\label{REMPersistence}
For a flow $\Phi$ and a partially hyperbolic set \(\Lambda\) for $\Phi$ with splitting  $T_\Lambda M=E^u\oplus E^c\oplus E^s$ and continuous invariant cone fields $\mathcal{C}^u, \mathcal{C}^s, \mathcal{C}^{cu}$ and $\mathcal{C}^{cs}$, containing  $E^u, E^s, E^{cu}$ and $E^{cs}$, respectively, there exist neighborhoods $U_0$ of \(\Lambda\) and $\mathcal{U}_0$ of $\Phi$ and cone fields $\mathcal{C}_0^u, \mathcal{C}_0^s, \mathcal{C}_0^{cu}$, and $\mathcal{C}_0^{cs}$ over $U_0$ such that if $\Psi\in\mathcal{U}_0$ and $\Lambda'\subset U_0$ is a compact \(\Psi\)-invariant set, then $\Lambda'$ is partially hyperbolic with a splitting $T_\Delta M=E^s_\Psi\oplus E^c_\Psi\oplus E^u_\Psi$ such that $E^u_\Psi$, $E^s_\Psi$, $E^{cu}_\Psi$, and $E^{cs}_\Psi$ are contained in $\mathcal{C}_0^u, \mathcal{C}_0^s, \mathcal{C}_0^{cu}$, and $\mathcal{C}_0^{cs}$, respectively.  

To avoid confusion we will sometimes refer to these neighborhoods as $U_0(\Phi, \Lambda)$ and $\mathcal{U}_0(\Phi, \Lambda)$ when we consider different flows and different partially hyperbolic sets.
\end{remark}
\begin{definition}[Accessibility]\label{DEFAccessibility}
Two points $p,q$ in a partially hyperbolic set $\Lambda\subset M$ are  \emph{accessible}\index{accessible points} if there are points $z_i\in M$ with $z_0=p$, $z_\ell=q$, such
that $z_i\in V^\alpha(z_{i-1})$ for $i=1,\dots,\ell$ and $\alpha=s$ or
$u$. The collection of points $z_0,z_1,\dots,z_\ell$ is called the
$us$-\emph{path} \index{us-path@\textit{us}-path} connecting \(p\) and $q$
and is denoted variously by $[p,q]_f=[p,q]=[z_0,z_1,\dots,z_\ell]$. (Note
that there is an actual path from \(p\) to $q$ that consists of pieces of
smooth curves on local stable or unstable manifolds with the $z_i$ as
endpoints.)

Accessibility is an equivalence relation and the collection of points 
accessible from a given point \(p\) is called the \emph{accessibility class}
\index{accessibility class} of \(p\).


A partially hyperbolic set \(\Lambda\) is \emph{bisaturated} if $W^u(x)\subset \Lambda$ and $W^s(x)\subset \Lambda$ for all $x\in \Lambda$,
 and a bisaturated partially hyperbolic set is said to be \emph{accessible}\index{accessible} if the accessibility class of any point is
the entire set, or, in other words, if any two points are
accessible.

If the entire manifold is partially hyperbolic for a flow, then it is bisaturated.  In this case,  the flow  is \emph{accessible} if the entire manifold is an accessibility class.

A pair $(\Phi, \Lambda)$ of a dynamical system and a partially hyperbolic set of it is \emph{accessible} on $X\subset M$ if for every $p\in X\cap \Lambda$ and $q\in X$ there is an $su$-path from \(p\) to $q$.  If $\Lambda$ is bisaturated, this implies that either $X\cap\Lambda=\emptyset$ or $X\subset\Lambda$.
Furthermore, a pair $(\Phi, \Lambda)$ of a dynamical system and a partially hyperbolic set of it is \emph{stably accessible} on $X\subset M$ if there exist neighborhoods $U$ of \(\Lambda\) and $\mathcal{U}$ of $\Phi$ such that if $\widetilde{\Phi}\in \mathcal{U}$ and $\widetilde\Lambda\subset U$ is a $\widetilde{\Phi}$-invariant bisaturated compact set, then $(\widetilde{\Phi}, \widetilde\Lambda)$ is accessible on $X$.
\end{definition}
Although we are interested in flows that are partially hyperbolic over the entire manifold to obtain our main results, our general result on accessibility (Theorem \ref{t.bisaturated}) holds for bisaturated partially hyperbolic sets.

We obtain accessibility (Theorem \ref{THMOpenDenseAccessibility}) by adapting from
\cite{DolgopyatWilkinson,AvilaCrovisierWilkinson} the proof of
\begin{theorem}[{Avila--Crovisier--Dolgopyat--Wilkinson \cite[Main Theorem]{DolgopyatWilkinson}, \cite[footnote p.\ 13]{AvilaCrovisierWilkinson}}]\label{THMAccessibility}
If \(M\) is a smooth compact manifold and \(r\ge1\), then stable accessibility
is \(C^1\) dense among\begin{itemize*}[before=\quad,itemjoin=\quad]
\item all
\item volume-preserving
\item symplectic
\end{itemize*}\quad
partially hyperbolic \(C^r\) diffeomorphisms of \(M\).
\end{theorem}
%
From Theorem \ref{THMOpenDenseAccessibility}, we obtain Theorem \ref{THMTrivFlowCentralizer} by adapting lemmas from the proof of
\begin{theorem}[{Burslem \cite[Theorem 1.2]{Burslem}}]\label{THMBurslem}
In the set of \(C^r\) partially hyperbolic diffeomorphisms of a compact manifold
\(M\) (\(r\ge1\)), there is a \(C^1\)-open and \(C^1\)-dense subset \(V\) whose elements all have
discrete diffeomorphism-centralizer.
\end{theorem}
\section{Accessibility}
The arguments of \cite{AvilaCrovisierWilkinson}  apply to the time-1 map of a partially hyperbolic flow---with one essential adaptation. We need to show that the perturbation that \(C^1\)-approximates a partially hyperbolic diffeomorphism by an accessible one \cite[Section 2.5]{AvilaCrovisierWilkinson} can be effected in such a way that it gives an accessible \emph{flow} as an approximation of a given partially hyperbolic \emph{flow}.

Theorem \ref{THMOpenDenseAccessibility} is a consequence of the following more general theorem, which corresponds to \cite[Theorem B]{AvilaCrovisierWilkinson}.

\begin{theorem}\label{t.bisaturated}
Let $\Lambda$ be a partially hyperbolic set for a flow $\Phi$ on a closed manifold $M$, and let $\mathcal{U}$ be a $C^1$ neighborhood of $\Phi$.  There exists a neighborhood $U$ of $\Lambda$ and a nonempty open set $\mathcal{O}\subset \mathcal{U}$ such that if $\Psi\in\mathcal{O}$ and $\Delta\subset U$ is a bisaturated partially hyperbolic set for $\Psi$, then $\Delta$ is accessible for $\Psi$.

Furthermore, this holds among volume-preserving, symplectic, and contact
flows.
\end{theorem}

\subsection*{Definitions}
We first review notation introduced in \cite{AvilaCrovisierWilkinson} before explaining the adaptations that need to be made to prove Theorems \ref{THMOpenDenseAccessibility} and \ref{t.bisaturated}.  We will use slightly different notation,  because we are using the notation for flows that they use for the charts.  
\begin{proposition}[Adapted charts]\label{PRPAdapted}
Let $M$ be a smooth manifold with $\dim(M)=d$.  For each point $p\in M$ there is a chart $f_p\colon B(0,1)\subset T_pM\to M$ with the following properties:
\begin{enumerate}
\item[(1)] The map $p\mapsto f_p$ is piecewise continuous in the $C^1$ topology.  So there are open sets $U_1,..., U_\ell \subset M$ and
\begin{itemize}
\item[--] compact sets $K_1,..., K_\ell$ covering $M$ with $K_i\subset U_i$,
\item[--] trivializations $g_i:U_i\times \mathbb{R}^d\to T_{U_i}M$ such that $g_i(\{p\}\times B(0,2))$ contains the unit ball in $T_pM$ for each $p\in U_i$, and
\item[--] smooth maps $F_i:U_i\times B(0,2)\to M$,
\end{itemize}
such that each $p\in M$ belongs to some $K_i$, with $f_p=F_i\circ g_i^{-1}$ on $B(0,1)\subset T_pM$.
\item[(2)] When a volume, symplectic, or contact form has been fixed on $M$, this pulls back under $f_p$ to a constant (and standard such) form on $T_pM$ \cite[Theorems 5.1.27, 5.5.9, 5.6.6]{KatokHasselblatt}.
\end{enumerate}
\end{proposition}
%
\begin{remark}
We note that given a compact set $K$ with continuous splitting $T_KM=E_1\oplus E_2$ there is a Riemannian metric with respect to which the norm of the projection from $E_2$ to $E_1$ is arbitrarily small.   Then the charts $F$ can be chosen so the bundles $E_1$ and $E_2$ are lifted in $B^d(0,3)$ to nearly constant bundles.
\end{remark}
\begin{remark}\label{REMDarboux}
For a volume, symplectic, or contact form, the standard chart expresses that form, respectively, as
\begin{itemize}
\item\(\displaystyle dx_1\wedge\dots\wedge dx_d\),
\item\(\displaystyle\sum_{i=1}^{d/2}dx_i\wedge dy_i\),
\item\(\displaystyle\alpha=dt+\sum_{i=1}^{(d-1)/2}x_idy_i\).
\end{itemize}
We note that a chart of the latter type is automatically of flow-box type: the Reeb vector field \(Y\) of  \(\alpha=dt+\sum_{i=1}^{(d-1)/2}x_idy_i\) is \(\partial/\partial t\) because it is (uniquely) defined by \(d\alpha(Y,\cdot)\equiv0\), \(\alpha(Y)\equiv1\).
\end{remark}

We do not assume dynamical coherence for the partially hyperbolic set \(\Lambda\) (the existence of a foliation tangent to $E^c$), and so we define approximate center manifolds that will be sufficient.  
\begin{definition}[$c$-admissible disk]
For sufficiently small $\eta>0$ and $p\in \Lambda$, denote by $B^c(0, \eta)$ the ball around 0 in $E^c_p$ of radius $\eta$. 
The set $V_\eta(p):= f_p(B^c(0, \eta))$ is a \emph{$c$-admissible disk} with radius $\eta=:r(V_\eta(p))$.  A \emph{$c$-admissible disk family} is a finite collection of pairwise disjoint, $c$-admissible disks.
\end{definition}
For $\beta\in (0,1)$ let $\beta V_\eta(p)$ be the $c$-admissible disk centered at \(p\) with radius $r(\beta V_\eta(p))=\beta\eta$.  
\begin{definition}[Return time]
The return time $R:\mathcal{P}(M)\to [0, \infty]$ is defined for a set $S\subset M$
that is contained in a flow-box (\cite[Definition 1.1.13]{FH}) of ``height'' \(\tau\) as the infimum of $t\in(\tau, \infty]$ such that $\varphi^t(S)\cap S\neq \emptyset$.
\end{definition}
Note that the above definition presumes that $S$ contains no fixed point of the flow. It implies that if $p\in M$ is not fixed, then $R(B_\eta(p))\lto{\eta\to0}\mathrm{per}(p)$ if we agree that $\mathrm{per}(p)=\infty$ if \(p\) is not periodic, and $\mathrm{per}(p)$ is the period of \(p\) for any periodic \(p\).

For a $c$-admissible disk family $\mathcal{D}$ and $\beta\in (0,1)$ we let
\[
\beta\mathcal{D}:=\{\beta D\,:\,D\in\mathcal{D}\},\qquad
 |\mathcal{D}|:=\bigcup_{D\in\mathcal{D}}D,\qquad
r(\mathcal{D}):=\sup_{D\in\mathcal{D}}r(D),\quad\textrm{and}\quad
R(\mathcal{D}):=R(|\mathcal{D}|).
\]
\subsection*{The Avila--Crovisier--Dolgopyat--Wilkinson arguments}
The proof of accessibility in \cite{AvilaCrovisierWilkinson} proceeds in two steps.  The first is a general fact for partially hyperbolic sets for diffeomorphisms on the existence of  $c$-admissible disk families
that stably meet all unstable and stable leaves as follows.  
\begin{definition}[Global $c$-section]
We say that a set $X\subset M$ is a \emph{(global) $c$-section} for $(\Phi, \Lambda)$ if $X\cap \Delta\neq \emptyset$ for every bisaturated subset $\Delta\subset \Lambda$.
\end{definition}
\begin{remark}
This terminology alludes to that of a (Poincar\'e) section for a flow, which meets bunches of orbits; (global) $c$-sections meet many stable and unstable leaves. Although this is not a defining property, the (global) $c$-sections we find will be transverse to stable and unstable leaves (Proposition \ref{p:stablesections}) and will indeed meet all stable and unstable leaves once accessible (Proposition \ref{p:stableaccessibility}).
\end{remark}
Via time-\(t\) maps, the result on the existence of such families immediately holds in our setting.
\begin{proposition}[{\cite[Proposition 1.4]{AvilaCrovisierWilkinson}}]\label{p:stablesections}
Let \(\Lambda\) be  partially hyperbolic set for \(\Phi\).  Then there exists a $\delta>0$ with the following property.  If $U$ is a neighborhood of \(\Lambda\) such that $\overline{U}\subset U_0(\Phi, \Lambda)$ and $T>0$, then there exists a $c$-admissible disk family $\mathcal{D}$ and $\sigma>0$ such that:
\begin{enumerate}
\item $r(\mathcal{D})< T^{-1}$,
\item $R(\mathcal{D})> T$ (this implies that \(|\mathcal{D}|\) contains no fixed point), and
\item if \, \(\Psi\) satisfies $d_1(\Phi, \Psi)<\delta$ and $d_0(\Phi, \Psi)<\sigma$, then for any bisaturated partially hyperbolic set $\Delta\subset U$ for \(\Psi\), the set $|\mathcal{D}|$ is a (global) $c$-section
 for $(\Psi, \Delta)$.
\end{enumerate}
\end{proposition}
\begin{remark}[Fixed points nowhere dense]\label{REMFPNWDense}
This prompts us to note that the set of fixed points of a partially hyperbolic dynamical system is nowhere dense: the set of fixed points of a continuous dynamical system is closed, and the restriction to it is the identity. The interior being nonempty is incompatible with partial hyperbolicity. Thus the $c$-admissible disk family in Proposition \ref{p:stablesections} can be chosen away from the set of fixed points. 
\end{remark}

The next step is a result about stable accessibility on center disks. Its proof needs adaptations for flows, the core part of which is Lemma \ref{l.perturbflow} below.
\begin{proposition}[{\cite[Proposition 1.3]{AvilaCrovisierWilkinson}}]\label{p:stableaccessibility}
If \(\Lambda\) is a partially hyperbolic set for a flow \(\Phi\) and $\delta>0$, then (with the notations of Remark \ref{REMPersistence}) there exist $T>0$ and a neighborhood $U$ of \(\Lambda\) such that $\overline{U}\subset U_0(\Phi, \Lambda)$ and if $\mathcal{D}$ is a $c$-admissible disk family with respect to $(\Phi, \Lambda)$ with $r(\mathcal{D})< T^{-1}$ and $R(\mathcal{D})>T$, then for all $\sigma>0$ there exists $\Psi\in \mathcal{U}_0(\Phi, \Lambda)$ such that:
\begin{enumerate} \setlength{\itemsep}{0pt}
\item $d_1(\Phi, \Psi)<\delta$,
\item $d_0(\Phi, \Psi)<\sigma$, and
\item if $D\in \mathcal{D}$ and $\Delta\subset U$ is a bisaturated partially hyperbolic set for \(\Psi\), then $(\Psi, \Delta)$ is stably accessible on $D$,
\item if \(\Phi\) preserves a volume, symplectic, or contact form, then so does \(\Psi\).
\end{enumerate}
\end{proposition}

Theorem \ref{t.bisaturated} follows from Propositions \ref{p:stablesections} and \ref{p:stableaccessibility} just as \cite[Theorem B]{AvilaCrovisierWilkinson} follows from \cite[Propositions 1.3 and 1.4]{AvilaCrovisierWilkinson} in \cite[Section 1.6]{AvilaCrovisierWilkinson}, including the preservation of a volume, symplectic, or contact form.

\begin{proof}[Proof of Proposition \ref{p:stableaccessibility}]
Much of the proof of this proposition is exactly as in \cite{AvilaCrovisierWilkinson}. We will explain the ideas in these parts while highlighting the points that need modifications for flows.

The first step \cite[Lemma 2.1]{AvilaCrovisierWilkinson} introduces smaller disks that are sufficiently close, but disjoint from, a $c$-admissible disk and have the property that there are $su$-paths at sufficiently small scales connecting any point in the $c$-admissible disk to one of the smaller disks, not just for the original map, but also for any maps that are sufficiently $C^1$ close.  

For these smaller disks one can then perform perturbations so that the flow is accessible on them.  Then one can show the perturbed flow will be accessible on $D$ (the $c$-admissible disk for the original flow).  By the choice of the $c$-admissible family we obtain accessibility of the bisaturated set for the perturbed flow.

The existence of the smaller disks does not use perturbations and assumes the existence of a $c$-admissible disk for a partially hyperbolic map.  This holds in our setting by considering time-\(t\) maps.  We include the statement of the result for completeness and adapt it for flows.

\begin{lemma}\label{l.smallerballs}
There exist $\delta_1, \rho_1>0$, $K>1$ and a neighborhood $U_1$ of \(\Lambda\) such that for any $\rho\in(0, \rho_1)$, any $c$-admissible disk $D$ with radius $\rho$, centered at $p\in \Lambda$, and for any $\epsilon\in (0, K^{-1}\rho)$, there exist $z_1,..., z_\ell\in T_pM$ such that:
\begin{enumerate}
\item[(1)] The balls $B(z_i, 100d^2\epsilon)$ are in the $K\epsilon$-neighborhood of $f_p^{-1}(D)$.
\item[(2)] The balls $B(z_i, 100d^2\epsilon)$ are pairwise disjoint.
\item[(3)] For any $x\in D$, there exists some $z_i$ such that for any \(\Psi\) that is $\delta_1$-close to \(\Phi\) in the $C^1$ distance and for any bisaturated set $\Delta\subset U_1$ for \(\Psi\):
\begin{enumerate}
\item[(a)] if $x\in \Delta$, then there is a $su$-path for \(\Psi\) between $x$ and $f_p(B(z_i, \epsilon))$,
\item[(b)] if $f_p(B(z_i, \epsilon))\subset \Delta$, then any point $y\in f_p(B(x, \epsilon/2))$ belongs to an $su$-path that intersects $f_p(B(z_i, \epsilon))$.
\end{enumerate}
\end{enumerate}
\end{lemma}

The idea of the proof of Theorem \ref{THMOpenDenseAccessibility} is to create small perturbations of the flow $\Phi$ supported near the points $z_1,..., z_\ell$ to create accessibility near each $z_i$.  This requires the following notion (which will be used in Lemma \ref{l.centeraccessibility} to show accessibility near the $z_i$).
\begin{definition}[$\theta$-accessibility]
A pair $(\Psi, \Delta)$ of a flow and a bisaturated set is \emph{$\theta$-accessible} on $f_p(B(z, 2d\epsilon)$ if there exist an orthonormal basis $w_1, ..., w_c$ of $E^c_p$ and for each $j\in \{1,..., c\}$ a continuous map
$$
H^j:[-1,1]\times[0,1]\times f_p^{-1}(\Delta)\cap B(z, 2d\epsilon)\to f_p^{-1}(\Delta)\cap B(0, 2\rho)$$
such that for any $x\in f_p^{-1}(\Delta)\cap B(z, 2d\epsilon)$ and $s\in [-1,1]$ we have
\begin{enumerate}
\item[(a)] $H^j(s,0,x)=x$,
\item[(b)] the map $f_p\circ H^j(s,., x):[0,1]\to \Delta$ is a 4-legged $su$-path (Brin quadrilateral), i.e., the concatenation of 4 curves, each contained in a stable or unstable leaf in alternation,
\item[(c)] $\|H^j(s,1,x)-x\| <\frac{\epsilon}{10d}$, and
\item[(d)] $\|H^j(\pm 1, 1 x)-(x \pm \theta \epsilon w j)\|<\theta \frac{\epsilon}{10d}$.
\end{enumerate}
\end{definition}
The second step is that $\theta$-accessibility in a neighborhood of a point implies accessibility on a smaller neighborhood.  This is a restatement of \cite[Lemma 2.2]{AvilaCrovisierWilkinson}.  The proof is again almost the same, but we provide it for completeness.

From now on write $d:=u + c + s$, where $\dim E^u=u$, $\dim E^c = c$, and $\dim E^s=s$.
\begin{lemma}\label{l.centeraccessibility}
For any $\theta>0$, there exist $\delta_2, \rho_2>0$ and a neighborhood $U_2$ of \(\Lambda\) such that 
\begin{enumerate}
\item [(1)] for any $p\in \Lambda$, any $z\in B(p, \rho_2)\subset T_pM$ and $\epsilon\in(0, \rho_2)$,
\item[(2)] for any flow \(\Psi\) that is $\delta_2$-close to \(\Phi\) in the $C^1$ topology,
\item[(3)] for any bi-saturated set $\Delta\subset U_2$ such that $(\Psi, \Delta)$ is $\theta$-accessible on $f_p(B(z, 2d\epsilon)$,
\end{enumerate}
the pair $(\Psi, \Delta)$ is accessible on $f_p(B(z,\epsilon))$.
\end{lemma}
\begin{proof}
We let $v_1,..., v_u$ be an orthonormal basis of $E^u_p$ and $v_{u+c+1},..., v_d$ an orthonormal basis for $E^s_p$.  We define local flows $\Phi_i$ on $f_p^{-1}(\Delta)$ as follows: let $X_i$ be a vector field along the leaves of $f_p^{-1}(W^j_\Psi)$  where $X_i(x)=D\pi^j_x(x+v_i)$ for $j=u$ if $1\leq i\leq u$ and $j=s$ if $u+c+1\leq j\leq d$, and the local flow is defined by the vector field $X_i$ on the set $B(0, 2\rho_1)\cap f_p^{-1}(\Delta)$ and $\rho_1$ is given by Lemma \ref{l.smallerballs}.  So the orbit of $x$ is the projection by $\pi^j_x$ on the curve $t\mapsto x + tv_i$ for $|t|<\rho_1$, and the orbits are $C^1$ curves whose tangent space is arbitrarily close to $\mathbb{R}v_i$ for sufficiently small constants $\rho_1$, $\delta_1$, and $U_1$ as in Lemma \ref{l.smallerballs}.

For $\rho_2$, $\delta_2$, and $U_2$ sufficiently small we see that 
\begin{equation}\label{e.close}
\|\varphi_i^t(x)-(x + t\theta \epsilon v_i)\|< |t|\theta \frac{\epsilon}{10d}.
\end{equation}
We also let $v_{u+j}=w_j$ be an orthonormal basis for the center direction and define inductively 
$$
\begin{array}{llll}
\varphi_{u+j}^t(x)=H^j(t, 1, x)\textrm{ when }t\in [0,1),\\
\varphi_{u+j}^t(x)=\varphi_{u+j}^{t-1}\circ \varphi_{u+j}^1(x)\textrm{ when }t>1,\\
\varphi_{u+j}^t(x)=\varphi_{u+j}^{t+1}\circ \varphi_{u+j}^{-1}(x)\textrm{ when }t<0
\end{array}
$$
where the above holds for $t$ so long as it can be defined.  From properties (c) and (d) in the definition of $H^j$ and estimate \ref{e.close} above we let $P(t_1,..., t_d)=\varphi_1^{t_1}... \varphi_d^{t_d}(x_0)$ for $(t_1,..., t_d)\in [-3\theta^{-1}, 3\theta^{-1}]^d$.  This is a continuous map and $\|P(t_1, ..., t_d)- (x_0 + \sum_i t_i \theta\epsilon v_i)\|<\frac{2\epsilon}{10}$.  The image of $P$ contains $B(x_0, \frac{5\epsilon}{2})$, and $f_p\circ P$ shows that $(\Psi, \Delta)$ is accessible on $B(z, \epsilon)$.
\end{proof}
\subsection*{The adaptation to flows}
We next produce the perturbations near the \(z_i\), from Lemma \ref{l.smallerballs}, that we need to establish accessibility.  This result and proof are similar to \cite[Lemma 2.3]{AvilaCrovisierWilkinson}, but in our case our perturbations need to be constructed for flows instead of maps. This is the essential adaptation of the Avila--Crovisier--Dolgopyat--Wilkinson arguments.
\begin{lemma}\label{l.perturbflow}
Consider a partially hyperbolic \emph{flow} \(\Phi\) generated by a vector field \(X\). With the previous notations,
there exist $\eta, \alpha_0>0$ such that for any $\alpha\in(0, \alpha_0)$, $p\in \Lambda$, $z\in B(0, 1/4)\subset T_pM$ with $X(f_p(z))\neq0$, $r\in (0, 1/4)$ and any unit vector $v\in E^c_p$ there is a vector field \(Y\) such that:
\begin{enumerate}
\item $Y=X$ outside $f_p(B(z,3r))$,
\item $df_p^{-1}Y=df_p^{-1}X+\alpha \eta  v$ on $B(z, 2r)$,
\item \(Y\) is $\alpha$-close to \(X\),
\item the flow \(\Psi\) defined by  \(Y\) is $\frac{r}{100d^2}$-close to \(\Phi\)
  in the $C^1$ distance.
\item if \(\Phi\) preserves a volume, symplectic, or contact form, then so does \(\Psi\).
\end{enumerate}
\end{lemma}
\begin{proof}
Figure \ref{FIGperturbation} (which utilizes that by Remark \ref{REMFPNWDense} we are working away from fixed points) illustrates what we would like to achieve: to connect the perturbed vector field (in the smallest circle) to the original one (in the square) by a bump-function interpolation.

For arbitrary vector fields this is all there is.
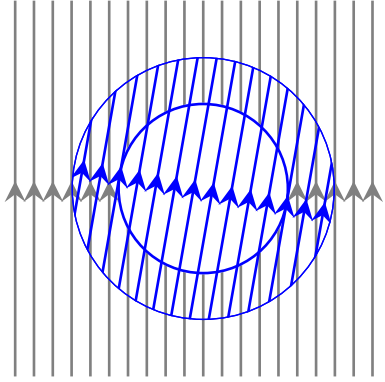
\begin{figure}[ht]
\scalebox{2.5}{\tikzset{->-/.style={decoration={markings,mark=at position #1 with {\arrow{>}}},postaction={decorate}}}
\begin{tikzpicture}[>=stealth]
\foreach \x in {-1,-.9,...,1}
\draw[->-=.52,gray](\x,-1)--(\x,1);
\draw[fill,white](0,0) circle (.45);
\begin{scope}
\clip(0,0) circle (.7);
\draw[blue](0,0) circle (.7);
\draw[blue](0,0) circle (.45);
\draw[blue](0,0) circle (.7);
\begin{scope}[rotate around={-10:(0,0)}]
\foreach \x in {-.95,-.85,...,.95}
\draw[->-=.52,blue](\x,-1)--(\x,1);
\end{scope}
\end{scope}
\end{tikzpicture}}
\caption{The perturbations in Lemma \ref{l.perturbflow}}\label{FIGperturbation}
\end{figure}
For volume-preserving flows, we invoke the pasting lemma \cite[Theorem 1]{Teixeira} (see also \cite{ArbietoMatheus}) to ensure volume-preservation of the perturbation.

For symplectic flows take symplectic flow-box charts as in Proposition \ref{PRPAdapted} in the rectangle and large circle in Figure \ref{FIGperturbation}; they are locally Hamiltonian with constant Hamiltonians on each neighborhood, so we can interpolate the Hamiltonians in the annulus.

For contact flows take a Darboux chart (Remark \ref{REMDarboux}) on the rectangle, then put a suitably rotated and scaled version in the large disk. This defines a local contact form whose Reeb field is as desired; interpolate the contact forms in the annulus. (For perturbations in the flow direction, i.e., reparameterization, no rotation is needed.)
\end{proof}

To construct the desired flow we will use the above perturbation to establish $\theta$-accessibility for the sets $f_{p_i}(B(z_i, 2d\epsilon))$.  We do this by adjusting Brin quadrilaterals by using the perturbation above.  Before explaining this step we first adjust the neighborhood $U$ and describe the setup we will need.

We first let $C_0^s$ and $C_0^u$ be cone fields in $U_0$ that are $\Phi$ invariant for $t<0$ and $t>0$, respectively.  Furthermore, from the choice of $\delta_3$ we know that the cone fields are invariant for any flow that is $\delta_3$-close in the  $C^1$ topology to \(\Phi\).  For $T>0$ we let $U$ be a neighborhood of \(\Lambda\) such that 
\begin{equation}\label{e.definitionofU}
\overline{U}\subset U_1 \cap U_2 \cap \bigcap_{|t|\leq T}\varphi^t(U_0).
\end{equation}
We also define $C^u=D\varphi^T(C_0^u)$ and $C^s=D\varphi^T(C_0^s)$ on $U$.  We know there exists some $T_1>0$ and $\rho_4>0$ such that if $T\geq T_1$ and $\rho<\rho_4$, then for any $p\in \Lambda$ we have $f_p(B(0, 2\rho))\subset U$ and the cone fields $Df_p^{-1}(C^s)$ and $Df_p^{-1}(C^u)$ on $B(0, 2\rho)$ are $\gamma$-close to $E^s_p$ and $E^u_p$ in $T_pM$ where $\gamma$ is smaller than $\alpha\eta$.

Let $\rho\in (0, \min\{ \rho_1, \rho_2, \rho_3, \rho_4\})$ and fix $T\geq T_1$ so that any $c$-admissible disk $D$ with center $p\in \Lambda$ and $r(D)<T^{-1}$ satisfies $f_p(D)\subset B(0, \rho)\subset T_pM$.  We also have $U$ defined by $T$ satisfying (\ref{e.definitionofU}) and the existence of a family $\mathcal{D}$ of $c$-admissible disks from Proposition \ref{p:stablesections} for some $\sigma>0$.

Fix $\theta=2\alpha\eta d$.  We now explain the quadrilaterals we will use to establish $\theta$-accessibility for the perturbed flow.  For $D\in \mathcal{D}$ we fix the $z_i$ as in Lemma \ref{l.smallerballs} and the sets $f_{p_i}(B(z_i, 100d^2\epsilon))\subset U$ for $\epsilon$ sufficiently small.  We can also define  subspaces $\mathcal{E}^s$ and $\mathcal{E}^u$  such that 
\begin{itemize}
\item $Df_{p_i}(z_i)\mathcal{E}^s\subset C^s(f_{p_i}(z_i))$,
\item $Df_{p_i}(z_i)\mathcal{E}^u\subset C^u(f_{p_i}(z_i))$,
\item $\dim\mathcal{E}^s=\dim E^s_{p_i}$, and 
\item $\dim\mathcal{E}^u=\dim E^u_{p_i}$.
\end{itemize}

Let $v_s\in \mathcal{E}^s$ and $v_u\in \mathcal{E}^u$ be unit vectors and fix $w_1,..., w_c$ an orthonormal basis for $E^c_{p_i}$.   For the foliations $\mathcal{F}^u$ and $\mathcal{F}^s$  we define the flows $\Phi'_k$ that corresponds to the linear flow $(x,t)\mapsto x+ tv_k$ projected to the leaves of $\mathcal{F}^k$ for $k\in\{u,s\}$. 
For each $j\in \{1,..., c\}$ we examine the quadrilaterals given by the composition
$$
v_{i,j}=\varphi'_s(-10jd\epsilon) \circ \varphi'_u(-10d\epsilon)\circ \varphi'_s(10jd\epsilon)\circ \varphi'_u(10d\epsilon)
$$
and
$$
v_{i,-j}=\varphi'_s(10jd\epsilon) \circ \varphi'_u(10d\epsilon)\circ \varphi'_s(-10jd\epsilon)\circ \varphi'_u(-10d\epsilon).
$$
We define \[R(i, \epsilon):=\max_{1\leq j\leq c}\big(\max (\|v_{i,j}\|, \|v_{i,-j}\|)\big).\]
Because \(R(i, \epsilon)\in o(\epsilon)\) \cite[equation (8)]{DolgopyatWilkinson}, we have:
\begin{proposition}\label{p.quadrilateral}
For each $z_i$ there exists some $\epsilon_0>0$ such that for $\epsilon\in(0, \epsilon_0)$ we have \(R(i, \epsilon)<\theta\epsilon/10d.\)
\end{proposition}


We let $\Psi$ be the flow generated by a vector field  whose restriction to $B(z_i + 10jd\epsilon v_s, 2d\epsilon)$  and $B(z_i - 10jd\epsilon v_s, 2d\epsilon)$ satisfy the conditions in
Lemma \ref{l.perturbflow}. 




Then 
starting at $z_i$ and using the quadrilaterals above we see that the the new flow coincides with  translation by $\theta\epsilon w_j$ from the original flow.  Indeed by construction first leg of the quadrilateral is left unperturbed by the new flow.  Similarly, the second leg is left unperturbed.  The third leg is the composition of $x\mapsto x-(10d\epsilon)v_u$ with the translation $\frac{\theta}{2} \epsilon w_j$.  The fourth leg similarly corresponds with the composition with the linear flow and translation by $ \frac{\theta}{2} \epsilon w_j$.  Then the quadrilateral on $B(z_i, 2d\epsilon)$ corresponds with translation by $\theta \epsilon w_j$, see Figure \ref{f.quadrilateral}. 
\begin{figure}[htb]
\begin{center}
\includegraphics[width=1.1\textwidth]{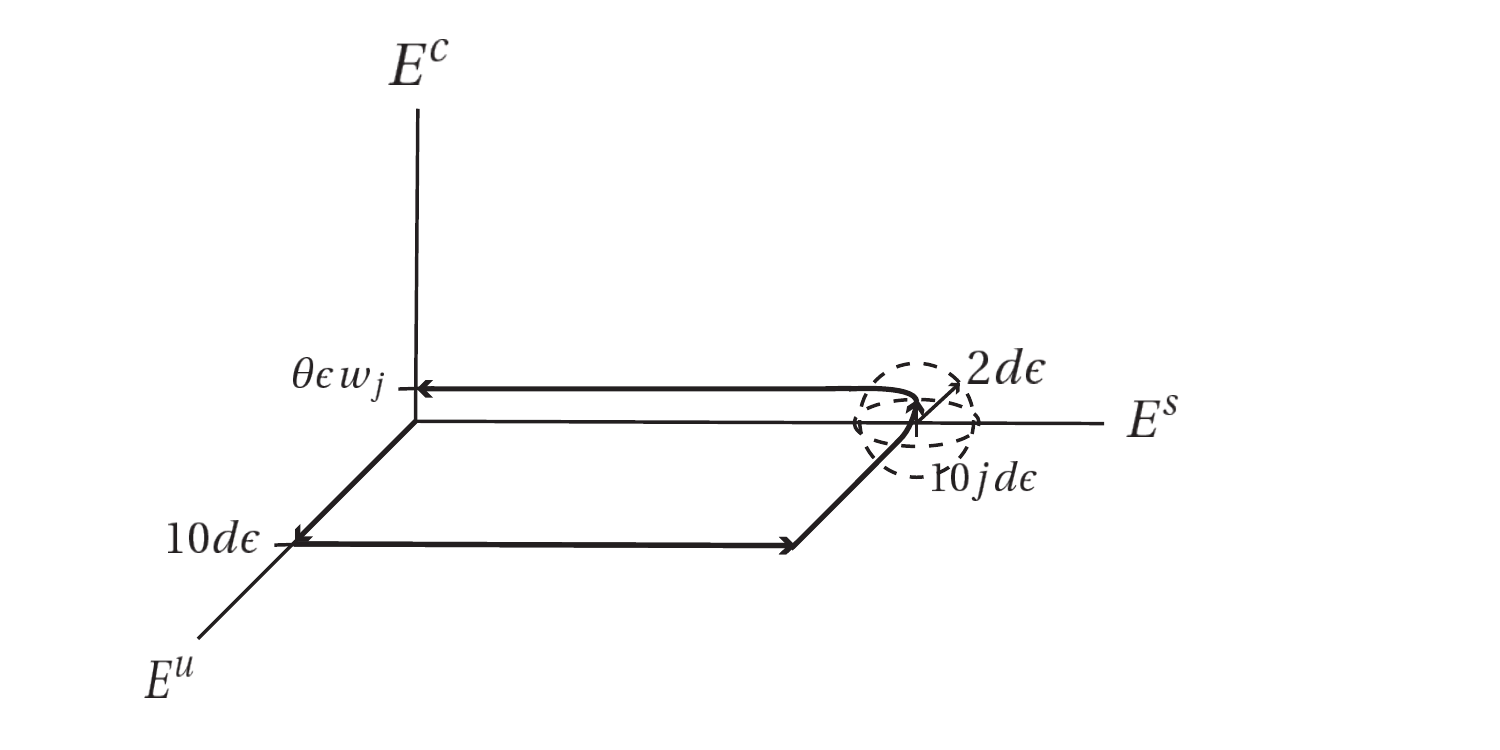}
\caption{Perturbed quadrilateral}\label{f.quadrilateral}
\end{center}
\end{figure}

Similarly, the quadrilateral associated with $-j$ corresponds with translation by $-\theta\epsilon w_j$ from the original flow.  From this we obtain the desired function $H^j$ used to ensure $\theta$-accessibility, see Figure \ref{f.thetaaccessible}.
\begin{figure}[htb]
\begin{center}
\includegraphics[width=.9\textwidth]{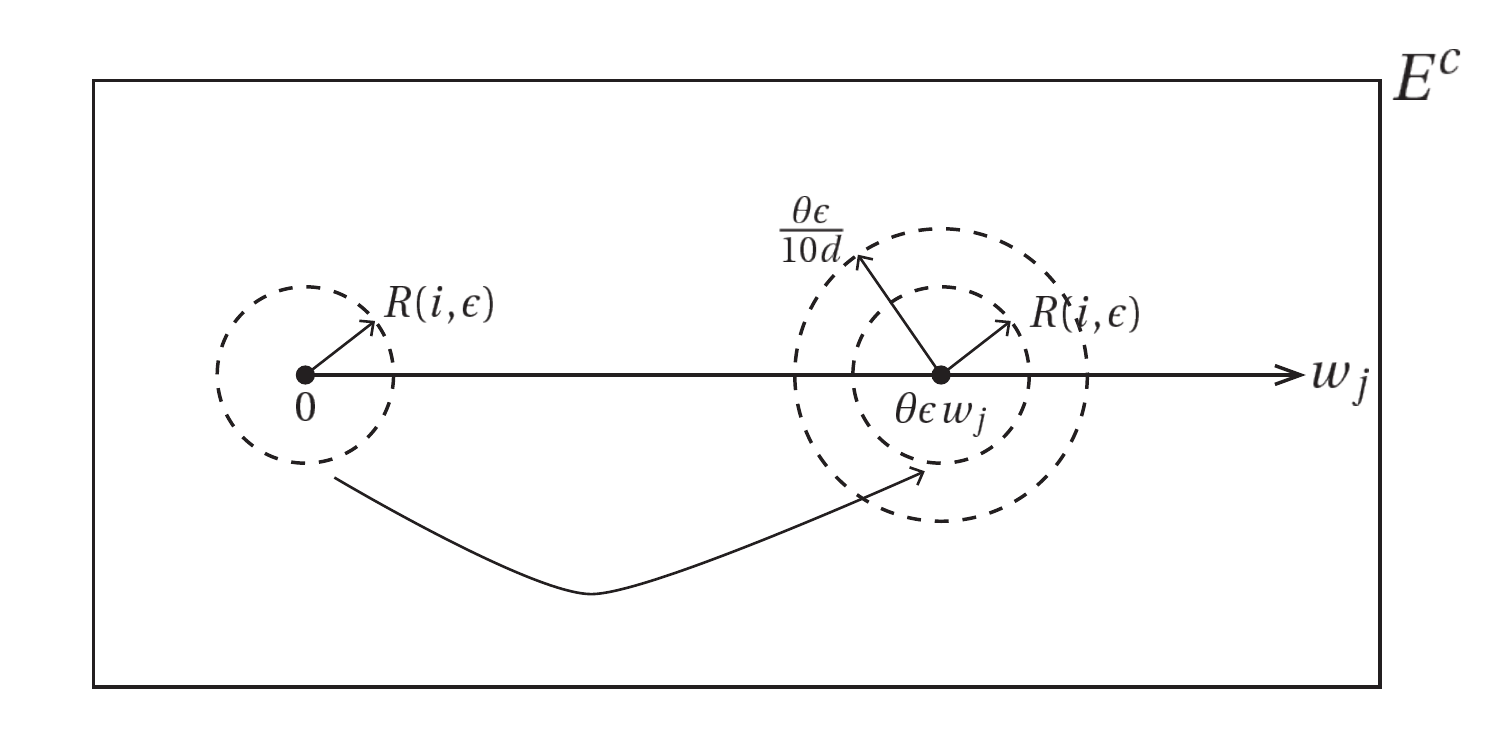}
\caption{$\theta$-accessibility}\label{f.thetaaccessible}
\end{center}
\end{figure}

Furthermore, from the size of the perturbations we see that the flows are at least $\frac{\epsilon}{100d}$-close in the $C^0$ topology.  This is the desired flow to prove Proposition \ref{p:stableaccessibility}.


Let $\Delta\subset U$ be a bisaturated set for \(\Psi\).  Let $\widetilde{\Psi}$ be $C^1$-close to \(\Psi\) and $\tilde{\Delta}\subset U$ be a  bisaturated set for $\widetilde{\Psi}$ contained in a small neighborhood of $\Delta$.  
By the construction above we know that $(\widetilde{\Psi}, \widetilde{\Delta})$ is $\theta$-accessible on each of the $f_{p_i}(B(z_i, 2d\epsilon))$.  Then $(\widetilde{\Psi}, \widetilde{\Delta})$ is accessible on each of the sets $f_{p_i}(z_i, \epsilon)$.

For $D\in \mathcal{D}$ that intersects $\widetilde{\Delta}$ at a point $z$ we know there exists a $su$-path for $\widetilde{\Psi}$ from $z$ to a point $y\in f_{p_i}(B(z_i, \epsilon))$ for some $i$ from Lemma \ref{l.smallerballs}.  
Furthermore, from Lemma \ref{l.smallerballs} we know if $x$ is $\epsilon/2$ close to $z$ in $D$, then there is a $su$-path to a point $y'\in f_{p_i}(B(z_i, \epsilon))$.  Accessibility on $B(z_i, \epsilon)$ implies there is a $su$-path from $y$ to $y'$.  So any point in the $\epsilon/2$-neighborhood of $z$ contained in $D$ is in the same accessibility class as $z$.  This implies that every point in $D$ is in the same accessibility class for $(\widetilde{\Psi}, \widetilde{\Delta})$ since $D$ is connected.  Then $(\Psi, \Delta)$ is stably accessible on any disk $D\in\mathcal{D}$. 
\end{proof}
%
%
%
\section{Centralizers}
With Theorem \ref{THMOpenDenseAccessibility} in hand, we now adapt arguments by Burslem from the proof of Theorem \ref{THMBurslem} \cite[Lemmas 5.2, 5.3]{Burslem} in order to prove our Theorem \ref{THMTrivFlowCentralizer}. We will use the following criterion, which holds for flows such that \(t\mapsto\varphi^t\) is invertible near \(t=0\) with its inverse continuous at the identity \cite[\S1.8]{FH}, but our statement instead invokes partial hyperbolicity for simplicity.
\begin{proposition}\label{PRPCentralizerCriterion}
A partially hyperbolic flow \(\Phi\) has trivial flow-centralizer (Definition \ref{DEFTrivialCentralizer}) if \(\Phi\) has discrete diffeomorphism-centralizer, i.e., if for any \(f\colon M\to M\), which is \(C^1\), commutes with \(\Phi\) and is sufficiently close to the identity, there is a \(\tau\) near 0 such that \(f=\varphi^\tau\).
\end{proposition}
\begin{proof}
If a flow \(\Psi\) commutes with \(\Phi\) then discreteness of the centralizer implies that \(\psi^s=\varphi^{\tau(s)}\) for small enough \(s\), hence also that \(\id=\psi^0=\varphi^{\tau(0)}\), i.e., \(\tau(0)=0\) (since \(\Phi\) is not periodic). 

Since \(D\varphi^{\tau(0)}(p)=\id=D\psi^0(p)=\lim_{s\to0}D\psi^s(p)=\lim_{s\to0}D\varphi^{\tau(s)}(p)\) (with \(p\) as above), partial hyperbolicity implies that \(\tau\) is continuous at 0, and the commutation relation gives \emph{additivity}: \(\tau(s+t)=\tau(s)+\tau(t)\). 

Together, these give a \(c\in\R\) such that \(\tau(s)=cs\) for all \(t\in\R\): Continuity at 0 implies continuity at any \(s\) because \(\tau(s+\epsilon)=\tau(s)+\tau(\epsilon)\lto{\epsilon\to0}\tau(s)\); additivity implies linearity of \(\tau\) on \(\mathbb Q\), then continuity implies linearity on \(\R\) \cite[Theorem 1, \S 2.1]{Aczel}, \cite[\S 2.1.1]{MR0124647}.
\end{proof}
We note that assuming the absence of fixed points makes a few of the arguments below a little simpler. 
\begin{proof}[Proof of Theorem \ref{THMTrivFlowCentralizer}] 
First, we produce a (nonfixed) closed orbit that is isolated among closed orbits of at most twice its period. 

There is a nonwandering point that is not fixed: every point in the support of an invariant Borel probability measure is nonwandering, and a partially hyperbolic flow has a nonatomic ergodic invariant measure because the topological entropy is positive \cite[Theorem 2]{CatsigerasTian} while atomic measures have zero entropy\rlap.\footnote{Lemma \ref{LEMNW} below bypasses this entropy argument---which is not needed when volume is preserved---at the expense of an initial perturbation.}

From this nonwandering point, the Pugh Closing Lemma gives a \(C^1\)-close (partially hyperbolic) flow with a closed orbit, and this works for volume-preserving or symplectic flows \cite{Arnaud}. 

Contact flows \emph{always} have a closed orbit by the Weinstein Conjecture \cite{Taubes}.

By the Transversality Theorem \cite[Theorems A.3.19, 7.2.4 \& p.\ 296]{KatokHasselblatt}, this closed orbit can (for a \(C^r\)-open \(C^1\)-dense set of such flows) be taken transverse and hence isolated among closed orbits of up to twice its period.

Theorem \ref{THMOpenDenseAccessibility} then gives a \(C^1\)-open dense set of accessible (volume-preserving/\hskip0pt{}symplectic/\hskip0pt{}contact) partially hyperbolic flows \(\Phi\) with a transverse closed orbit \(\Oc(p)\) that is (hence) isolated among closed orbits of at most twice its period. 

We now verify that their diffeomorphism-centralizer is discrete (Proposition \ref{PRPCentralizerCriterion}). 

If \(f\colon M\to M\) is \(C^1\), commutes with \(\Phi\) and is sufficiently close to the identity, then \(f\) maps closed orbits of \(\Phi\) to closed orbits of \(\Phi\) with the same period. Thus, since \(\Oc(p)\) is isolated among orbits of the same period, \(f(p)\in\Oc(p)\) once \(f\) is sufficiently close to the identity, and indeed \(f(p)=\varphi^\tau(p)\) for some \(\tau\) near 0. Thus \(h\dfn f\circ\varphi^{-\tau}\) fixes \(p\) and commutes with \(\Phi\). Here, ``sufficiently close'' means that \(d_{C^0}(h,\id)<\epsilon\), where \(\epsilon>0\) is as in Lemma \ref{LEMBurslem52} below. 

We wish to see that this implies that \(h=\id\), and it suffices to verify
this on the dense set \(\mathcal M\) of points that are accessible from
\(p\) with \((su)\)-paths that avoid fixed stable and unstable leaves,
i.e., disjoint from \(W^u(x)\) and \(W^s(x)\) for any fixed point \(x\). (This set is dense because the set of fixed points is finite and the invariant manifolds have positive codimension.)

For any \(y\in\mathcal M\), recursive application of Lemma \ref{LEMBurslem53} below to a finite \textsl(su)-path from \(p\) to \(y\) shows that \(h\) fixes all vertices of this path and hence \(y\). 
\end{proof}
\begin{lemma}[{\cite[Lemma 5.3]{Burslem}}]\label{LEMBurslem53}
If \(h(q)=q\) and \(W^u(q)\) contains no fixed point, then \(h(x)=x\) for all \(x\in W^u(q)\). Likewise for \(W^s(q)\).
\end{lemma}
\begin{proof}
Suppose \(x\in W^u(q)\); the case  \(x\in W^s(q)\) is analogous. Then
\[
h(x)\in h(W^u(q))\xeq{h\in C^1}W^u(h(q))\xeq{h(q)=q}W^u(q)=W^u(x),
\]
while
\[
d(\varphi^t(x),\varphi^t(h(x)))=
d(\varphi^t(x),h(\varphi^t(x)))
<d_{C^0}(h,\id)<\epsilon.
\]
This implies \(h(x)=x\) by Lemma \ref{LEMBurslem52} below.
\end{proof}
The proof of \cite[Lemma 5.2]{Burslem} applies to leaves containing no fixed points:
\begin{lemma}
\label{LEMBurslem52}
There is an \(\epsilon > 0\) such that if \(W^u(x)\) contains no fixed point and \(y\in W^u(x)\) satisfies \(d(\varphi^t(x),\varphi^t(y))<\epsilon\) for all \(t\in\R\), then \(x = y\). Likewise for \(W^s(x)\).
\end{lemma}
As promised, we provide here an alternate argument for the existence of the nonfixed nonwandering point.
\begin{lemma}\label{LEMNW}
A partially hyperbolic flow can be \(C^1\)-perturbed to have a nonwandering point that is not fixed.
\end{lemma}
\begin{proof}
For volume-preserving flows, all points are nonwandering by the Poincar\'e Recurrence Theorem, so there are nonwandering points which are not fixed.

Without volume-preservation, the Kupka--Smale Theorem \cite[Theorem 6.1.6]{FH} gives a \(C^r\) perturbation for which all fixed points are hyperbolic (and hence also finite in number). Unless there is an additional nonwandering point, the nonwandering set is then hyperbolic and equal to the limit set \cite[Proposition 1.5.17]{FH}---so the union of their (finitely many!) stable manifolds is \(M\) \cite[Proposition 5.3.40]{FH}, contrary to partial hyperbolicity (stable manifolds have positive codimension).
\end{proof}
\bibliography{centralizerbib}{}
\bibliographystyle{plain}
\end{document}